\documentclass[12pt,a4paper]{amsart}

\usepackage{amsmath,amsthm,amsfonts,amssymb,array,amscd,color,lineno}
\usepackage{amsmath,geometry,amssymb,amsfonts,amsthm,graphicx,enumerate,latexsym,tabularx,amscd}
\usepackage[all,cmtip]{xy}
\usepackage{refcheck} 
\norefnames
\nocitenames
\usepackage{geometry}
\usepackage{amsthm}
 \newtheorem{thm}{Theorem}[section]
 \newtheorem{lem}[thm]{Lemma}
 
 \newtheorem{cor}[thm]{Corollary}
\theoremstyle{definition}
 \newtheorem{rem}[thm]{Remark}
 \newtheorem{conj}{Conjecture}
 \numberwithin{equation}{section}
\theoremstyle{definition}
\theoremstyle{remark}
 \numberwithin{equation}{section}

\usepackage{tikz}
\usepackage[pdftex]{hyperref}
\usetikzlibrary{matrix,arrows}

\usepackage[all]{xy}

\newcommand{\F}{\mathbb{F}}

\newcommand{\ve}{\varepsilon}

\newcommand{\ol}{\overline}

\newcommand{\bea}{\begin{eqnarray}}
\newcommand{\eea}{\end{eqnarray}}

\newcommand{\bbC}{\mathbb{C}}

\newcommand{\bbQ}{\mathbb{Q}}
\newcommand{\bbR}{\mathbb{R}}

\newcommand{\bbZ}{\mathbb{Z}}   

\renewcommand{\and}{\quad \mbox{and} \quad}  
\renewcommand{\le}{\leqslant}\renewcommand{\leq}{\leqslant}
\renewcommand{\ge}{\geqslant}\renewcommand{\geq}{\geqslant}

\title{Epsilon factors of symplectic type characters in the wild case }

\subjclass[2010]{11S37; 22E50}

\keywords{Local field, Epsilon factor, Symplectic type character, Conductor}

\author[Biswas]{\bfseries Sazzad Ali Biswas}

\address{
Einstein Institute of Mathematics\\ 
Hebrew University of Jerusalem\\ 
Givat Ram, Jerusalem, 91904, Israel}
\email{sazzad.biswas@mail.huji.ac.il, sazzad.jumath@gmail.com}

\thanks{This research was supported by the Israel Science Foundation (grant no: 1676/17)} 

\begin{document}

\vspace{10mm}
\setcounter{page}{1}
\thispagestyle{empty}


\begin{abstract}
By work of John Tate we can associate an epsilon factor with every multiplicative character of a local field.
In this paper we determine the explicit signs of the epsilon factors for symplectic type characters of $K^\times$,
where $K/F$ is a wildly ramified quadratic extension of a non-Archimedean local field $F$ of characteristic
zero.

\end{abstract}

\maketitle

\section{\textbf{Introduction}}

Let $F$ be a non-Archimedean local field of characteristic zero, i.e., 
a finite extension of $p$-adic numbers field $\mathbb{Q}_p$, where $p$ is a prime. 
Let $K/F$ be a finite extension of the field $F$.
Let $e_{K/F}$ be the ramification index of the extension $K/F$ and $f_{K/F}$ be 
the residue degree of $K/F$. The extension $K/F$ is called \textbf{unramified} if $e_{K/F}=1$;
equivalently $f_{K/F}=[K:F]$. 
The extension $K/F$ is called \textbf{totally ramified} if 
$e_{K/F}=[K:F]$; equivalently $f_{K/F}=1$. 
Let $\mathrm{Char}(\mathbb{F}_{q_F})$ be the characteristic of the residue field 
$\mathbb{F}_{q_F}$ of $F$, where 
$q_F$ is the cardinality of the residue field of $F$. 
If $\mathrm{gcd}(\mathrm{Char}(\mathbb{F}_{q_F}),[K:F])=1$, then the extension 
$K/F$ is called \textbf{tamely ramified}, otherwise \textbf{wildly ramified}. 
A multiplicative character $\chi$ of $K^\times$ (where $K/F$ is a quadratic field extension) is said to be 
\textbf{symplectic type} if $\chi|_{F^\times}=\omega_{K/F}$, where $\omega_{K/F}$ is the quadratic 
character of $F^\times$ associated
to $K^\times$ by class field theory.

In \cite{DP}, Dipendra Prasad shows that the epsilon factor of a symplectic type character is $\pm 1$ and in
the unramified and totally tamely ramified cases, he gives explicit sign of symplectic type character
(cf. \cite{DP}, pp. 22-23) by 
using Deligne's twisting formula for characters (cf. Corollary \ref{Corollary 2.2}(3)). This explicit computation of the 
{\it signs} has many applications. For instance, in Theorem 1.2 of \cite{DP}, Prasad generalizes the 
Tunnell Theorem (cf. \cite{Tunnell}) in terms of the explicit signs of symplectic type characters. As to application, 
one also can see  
\cite{MT}.
But in the {\bf wild case} (that is, when $K/F$ is a quadratic {\it wildly} ramified extension),
the determination of the explicit signs of 
epsilon factors of symplectic type characters were not known. 
In this paper, we determine the explicit signs of the epsilon factors
of symplectic
type characters in the wild case.
The computation of this paper 
is based on
the {\bf Lamprecht-Tate formula} (cf. \cite{JT}, Proposition 1). 

To apply the  Lamprecht-Tate formula, we first need to know the conductor of characters.
In the  following theorem first we show that the conductor of a symplectic type 
character is either $2t+1$ or a even $\geq 2(t+1)$, where
$t$ is the ramification break of the extension $K/F$ (or of the Galois group $\mathrm{Gal}(K/F)$).
\begin{thm}\label{Theorem 1.1}
Let $K/F$ be a quadratic wildly ramified extension of $F$ and $t$ be  the ramification break of 
 the extension $K/F$. Denote by $\omega_{K/F}$ as the quadratic character of $F^\times$ associated to $K$
 by class field theory.  Let $m$ be the conductor of symplectic type character $\chi$. Then 
the conductor of $\omega_{K/F}$ is $t+1$ and 
 $m=2t+1$ when $m$ is odd and $m\geq2(t+1)$ when $m$ is even.

\end{thm}

When the conductor of a symplectic type character is {\bf even}, we have the following theorem.
 \begin{thm}\label{Theorem 1.2}
Let $K/F$ be a wildly ramified quadratic extension of $F$. Let $\psi_F$ be a fixed nontrivial additive character of $F$.
Let $\pi_K$ be a uniformizer of $K$ and consider $\pi_F=N_{K/F}(\pi_K)$ as 
a uniformizer of $F$.
Let $\psi$ be a nontrivial additive character of $K$ of the form 
$\psi(x):=c\cdot(\psi_F\circ Tr_{K/F})(x)=\psi_F(Tr_{K/F}(cx))$ for $x\in K$ with conductor
$n(\psi)=0$,
where $c\in K^\times$ such that the trace $Tr_{K/F}(c)=0$.
Let $\chi$ be a symplectic type character of $K^\times$ with even conductor $a(\chi)=2d$ where $d\ge 2$. 
Then we have
\begin{equation}
 \epsilon(\chi,\psi)=\chi(-1)^d.
\end{equation}
\end{thm}
And when the conductor of the symplectic type character is $2t+1$, we obtain the following theorem.
\begin{thm}\label{Theorem 1.3}
Let $K/F$ be a wildly ramified quadratic extension of $F$. Let $\psi_F$ be a fixed nontrivial additive character of $F$.
Let $\pi_K$ be a uniformizer of $K$ and consider $\pi_F=N_{K/F}(\pi_K)$ as 
a uniformizer of $F$.
Let $\psi$ be a nontrivial additive character of $K$ of the form $\psi=c\cdot(\psi_F\circ Tr_{K/F})$ with conductor 
$n(\psi)=2l+1$,
where $c\in K^\times$ such that the trace $Tr_{K/F}(c)=0$.
Let $t$ be the ramification break of the extension $K/F$.
Let $\chi$ be a symplectic type character with conductor $a(\chi)=2t+1$.
 Then we have
\begin{equation}
              \epsilon(\chi,\psi)=\chi(-1)^l G(Q),
             \end{equation}
where 
$$G(Q):= q_K^{-1/2}\sum_{x\in P_K^t/P_K^{t+1}} Q(x), \quad Q(x):=\chi^{-1}(1+x)(c'^{-1}\psi)(x),$$
and $\nu_K(c')=2t+2l+2$ such that
$$\chi(1+y)=\psi(\frac{y}{c'}),\quad\text{for all $y\in P_{K}^{t+1}$}.$$

\end{thm}


\begin{rem}[{\bf Convention for epsilon factors}]
Mainly there are two conventions for local epsilon factors. They are due to Langlands (cf. \cite{RL}) and Deligne (cf. \cite{D1}). 
In this paper, we use Langlands' convention (cf. \cite{JT2}, p. 17, Equation (3.6.4)) 
$\epsilon_L(\chi,\psi,s)$, $s\in\bbC$ for epsilon factor,
and our $\epsilon(\chi,\psi)$ is as follows:
\begin{center}
 $\epsilon(\chi,\psi):=\epsilon_{L}(\chi,\psi,\frac{1}{2})$.
\end{center}

 
\end{rem}

\section{\textbf{Notations and Preliminaries}}

Let $F$ be a non-Archimedean local field of characteristic zero. 
Let $O_F$ be the ring of integers of $F$, $P_F$ the unique prime ideal of $O_F$ and $\pi_F$ a choice of 
uniformizer. Let $q_F$ be the cardinality of the residue field.
Let $U_F=O_F-P_F$ be the group of units in $O_F$.
Let $P_{F}^{i}=\{x\in F:v_F(x)\geq i\}$ and for $i\geq 0$ define $U_F^i=1+P_{F}^{i}$
(with $U_{F}^{0}=U_F=O_{F}^{\times}$).

We let $a(\chi)$ be the conductor of 
 nontrivial character $\chi: F^\times\to \mathbb{C}^\times$, i.e., $a(\chi)$ is the smallest
 integer $\geq 0$ such that $\chi$ is trivial
 on $U_{F}^{a(\chi)}$. We call $\chi$ is {\it unramified} if the conductor of $\chi$ is zero and otherwise {\it ramified}.
 If $K/F$ is a quadratic field extension of $F$, then we denote 
 $\omega_{K/F}$ as the quadratic character of $F^\times$ associated to $K$ by class field theory; i.e., it is 
 a unique nontrivial character of $F^\times/N_{K/F}(K^\times)$, where $N_{K/F}$ denotes the norm map 
 from $K^\times$ to $F^\times$.


The {\it conductor} of any nontrivial additive character $\psi$ of the field $F$ is an integer
$n(\psi)$ if $\psi$ is trivial
on $P_{F}^{-n(\psi)}$, but nontrivial on $P_{F}^{-n(\psi)-1}$. 

\subsection{Epsilon factors}

For a nontrivial multiplicative character $\chi$ of $F^\times$ and nontrivial additive character
$\psi$ of $F$, we have 
\begin{equation}\label{eqn 2.1}
 \epsilon(\chi,\psi)=\chi(c)\frac{\int_{U_F}\chi^{-1}(x)\psi(x/c) dx}{|\int_{U_F}\chi^{-1}(x)\psi(x/c) dx|},
\end{equation}
where the Haar measure $dx$ is normalized such that the measure of $O_F$ is $1$ and 
 $c\in F^\times$ with F-valuation $n(\psi)+a(\chi)$. The above formula (\ref{eqn 2.1}) can be modified 
(cf. \cite{JT}, pp. 93-94) as follows:
\begin{equation}\label{eqn 2.2}
 \epsilon(\chi,\psi)=\chi(c)q_F^{-a(\chi)/2}\sum_{x\in\frac{U_F}{U_{F}^{a(\chi)}}}\chi^{-1}(x)\psi(x/c).
\end{equation}
where $c=\pi_{F}^{a(\chi)+n(\psi)}$.
\begin{rem}
If $u\in U_F$ is a unit and replace $c=cu$ in equation (\ref{eqn 2.2}), then we observe that 
$\epsilon(\chi,\psi)$ \textbf{depends only} on the exponent $v_{F}(c)=a(\chi)+n(\psi)$.

\end{rem}

For this paper we need the following Lamprecht-Tate formula.

\begin{thm}[\cite{JT}, Proposition 1, Lamprecht-Tate formula]\label{Theorem 2.1}
Let $F$ be a non-Archimedean local field.
Let $\chi$ be a character of $F^\times$ of conductor $a(\chi)$ and let $m$ be a natural number 
such that $2m\le a(\chi)$. Let $\psi$ be a nontrivial additive character of $F$. 
Then there exists $c\in F^\times$, $\nu_F(c)=a(\chi)+n(\psi)$ such that 
\begin{equation}\label{eqn 5.4.5}
 \chi(1+y)=\psi(c^{-1}y)\qquad\text{for all $y\in P_{F}^{a(\chi)-m}$},
\end{equation}
and for such a $c$ we have:
\begin{equation}\label{eqn 6.0.9}
 \epsilon(\chi,\psi)=\chi(c)\cdot q_{F}^{-\frac{(a(\chi)-2m)}{2}}
 \sum_{x\in U_F^m/U_F^{a(\chi)-m}}\chi^{-1}(x)\psi(c^{-1}x).
\end{equation}
{\bf Remark:} The assumption (\ref{eqn 5.4.5}) is obviously fulfilled for $m=0$ because then both sides are $=1$,
and the resulting formula for $m=0$ is the original formula (\ref{eqn 2.2}) for abelian local constant.
\end{thm}




\begin{cor}[cf. \cite{RL}, Lamprecht Formulas, Lemma 8.1]\label{Corollary 2.2}
 Let $\chi$ be a character of $F^\times$. Let $\psi$ be a nontrivial additive character of $F$.
\begin{enumerate}
 \item When $a(\chi)=2d\, (d\ge 1)$, we have 
$$\epsilon(\chi,\psi)=\chi(c)\psi(c^{-1}).$$
\item When $a(\chi)=2d+1\, (d\ge1)$, we have 
$$\epsilon(\chi,\psi)=\chi(c)\psi(c^{-1})\cdot q_F^{-\frac{1}{2}}\sum_{x\in P_F^d/P_F^{d+1}}\chi^{-1}(1+x)\psi(c^{-1}x).$$
\item {\bf Deligne's twisting formula (cf. \cite{D1}, Lemma 4.16):}
If $\alpha, \beta\in \widehat{F^\times}$ with $a(\alpha)\ge 2\cdot a(\beta)$, then 
$$\epsilon(\alpha\beta,\psi)=\beta(c)\cdot \epsilon(\alpha,\psi).$$
\end{enumerate}
Here $c\in F^\times$ with $F$-valuation $\nu_F(c)=a(\chi)+n(\psi)$, and in Case (1) and Case (2), $c$ also satisfies 
\begin{center}
 $\chi(1+x)=\psi(\frac{x}{c})$ for all $x\in F^\times$ with $2\cdot\nu_F(x)\ge a(\chi)$.
\end{center}
And in case (3), we have $\alpha(1+x)=\psi(x/c)$ for all $\nu_F(x)\ge \frac{a(\alpha)}{2}$.
\end{cor}

\subsection{Ramification break}

Let $K/F$ be a Galois extension of $F$ and $G$ be the Galois group of the extension $K/F$. For each $i\ge -1$ we define the 
$i$-th ramification subgroup of $G$ (in the lower numbering) as follows:
\begin{center}
 $G_i=\{\sigma\in G|\quad v_K(\sigma(\alpha)-\alpha)\geq i+1\quad \text{for all $\alpha\in O_K$}\}$.
 \end{center}
 An integer $t$ is called a {\bf ramification break or jump} for the extension $K/F$ or the ramification groups 
 $\{ G_i\}_{i\ge -1}$ if 
 $$G_t\ne G_{t+1}.$$
We also know that there is a decreasing filtration (with upper numbering) of $G$ and which is defined by the 
{\bf Hasse-Herbrand} function $\Psi=\Psi_{K/F}$ as follows:
$$G^u=G_{\Psi(u)}, \quad\text{where $u\in\bbR,\, u\ge -1$}.$$
Since by the definition of Hasse-Herbrand function, $\Psi(-1)=-1, \Psi(0)=0$, we have
$G^{-1}=G_{-1}=G$, and $G^0=G_0$. Thus a real number $t\ge -1$ is called a ramification break for $K/F$ or the filtration
$\{G^i\}_{i\ge -1}$ if 
$$G^t\ne G^{t+\varepsilon},\quad\text{for all $\varepsilon> 0$}.$$
When $G$ is {\bf abelian}, it can be proved (cf. {\bf Hasse-Arf theorem}, \cite{FV}, p. 91) that the ramification
breaks for $G$ are {\bf integers.} But in general, the set of ramification breaks of a Galois group of a local field is 
{\it countably infinite and need not consist of integers.}

 If $K/F$ is quadratic extension, it can be proved that there exists a unique
 ramification break $t$ for which
 we have 
 $$G_t=G\quad and \quad G_{t+1}=\{1\}.$$
 When $K/F$ is unramified (resp. totally tamely ramifield), the ramification jump is $t=-1$ (resp. jump $t=0$). And 
 when $K/F$ is quadratic wildly ramifield extension, the ramification jump is $t\ge 1$.

\section{\textbf{Epsilon factors of Symplectic characters }}

\begin{lem}\label{Lemma 3.1}
 Let $K/F$ be a quadratic wildly ramified extension of $F/\bbQ_2$ and let $\psi_F$ be a fixed nontrivial additive 
 character of $F$.
 Let $\psi$ be a nontrivial character of $K$ of the form 
 $$\psi:=c\cdot(\psi_F\circ Tr_{K/F}),\quad \text{where $c\in K$ with $Tr_{K/F}(c)=0$}.$$
 Then $\psi$ is trivial on $F$ and the conductor of $\psi$ is:
\begin{equation}\label{eqn 3.1}
 n(\psi)=2\cdot n(\psi_F)+\nu_K(c)+d_{K/F},
\end{equation}
 where $d_{K/F}$ is the exponent of the different $\mathcal{D}_{K/F}$ of the extension $K/F$.
\end{lem}

\begin{proof}
 We know that (cf. \cite{AW}, p. 142, Corollary 3) the conductor of $\psi$ is 
\begin{align*}
 n(\psi)=n(c\cdot(\psi_F\circ Tr_{K/F}))=\nu_K(c)+n(\psi_F\circ Tr_{K/F})\\
 =\nu_K(c)+e_{K/F}\cdot n(\psi_F)+d_{K/F}=\nu_K(c)+2\cdot n(\psi_F)+d_{K/F}.
\end{align*}
 Again by the given condition, we have $Tr_{K/F}(c)=0$. Therefore, for any $x\in F$, we can show that:
 $$\psi(x)=c\cdot (\psi_F\circ Tr_{K/F})(x)=\psi_F(Tr_{K/F}(xc))=\psi_F(0)=1.$$
This proves that $\psi$ is trivial on $F$.
 \end{proof}

\begin{rem}
{\bf i).} Conversely, if we choose a nontrivial character $\psi$ which is trivial on $F$, then it can be proved that $\psi$ is of the 
form 
$$\psi=c\cdot (\psi_F\circ Tr_{K/F}),$$
where $c\in K$ with $Tr_{K/F}(c)=0$ and $\psi_F$ is some suitable nontrivial additive character of $F$.


Furthermore from Lemma \ref{Lemma 3.1}, we can write:
$$n(\psi)=n(\psi_F\circ Tr_{K/F})+\nu_K(c)=2\cdot n(\psi_F)+d_{K/F}+\nu_K(c)\equiv d_{K/F}+\nu_K(c)\pmod{2}.$$
But  $[K:F]=2$ implies that  $Ker(Tr_{K/F})$ is a  1-dimensional  $F$-space. Thus any other  $c'$
has the form   $c'=ac$ where $a\in F$, hence  
$$\nu_K(c')=\nu_K(a)+\nu_K(c)=2\nu_F(a)+\nu_K(c)\equiv \nu_K(c)\pmod{2}.$$
Therefore   $\nu_K(c')\pmod{2}$   does not depend on the choice of  c.

{\bf ii).} Again by suitable choice of $\psi_F$ and $c$, we can 
always construct (by using Theorem 2.6 on p. 199 of \cite{JN} and equation (\ref{eqn 3.1})) 
such nontrivial additive characters $\psi$ which are taken 
in Theorem \ref{Theorem 1.2} and Theorem \ref{Theorem 1.3}.
\end{rem}

 \begin{proof}[{\bf Proof of Theorem \ref{Theorem 1.1}}]
  By local class field theory, we have $\mathrm{Gal}(K/F)\cong F^{\times}/N_{K/F}(K^\times)$ and $\omega_{K/F}$ is the quadratic character
  of $F^{\times}/N_{K/F}(K^\times)$, i.e., $\omega_{K/F}$ is trivial on $N_{K/F}(K^\times)$. We also know that 
  \begin{equation}\label{eqn 6.11}
   U_{K}^{n}\cap F^\times=\begin{cases}
                           U_{F}^{\frac{n}{2}} & \text{when $n$ is even},\\
                           U_{F}^{\frac{n+1}{2}} & \text{when $n$ is odd}.
                          \end{cases}
\end{equation}
Let $s$ be the conductor of the quadratic character $\omega_{K/F}$. Then from the definition of conductor
we can say 
$\omega_{K/F}$ is trivial on $U_{F}^{s}$, but nontrivial on $U_{F}^{s-1}$. The character $\omega_{K/F}$ is also trivial on norm 
$N_{K/F}(K^\times)$. By Corollary 3 on p. 85 of \cite{JP}, we have 
$$N_{K/F}(U_{K}^{\varPsi(n)})=U_{F}^{n}\quad for \quad n>t\quad and \quad N_{K/F}(U_{K}^{\varPsi(n)+1})=U_{F}^{n+1}
\quad for\quad n\ge t.$$
This implies $\omega_{K/F}$ trivial on $U_{F}^{n+1}$ when $n\geq t$
and trivial on $U_{F}^{n}$ when $n>t$. Therefore $t+1\geq s$. Now we have show that $s=t+1$. We assume the conductor of
$\omega_{K/F}$ is greater than or equal to $t+2$, therefore 
$\omega_{K/F}$ is nontrivial on $U_{F}^{t+1}$ which is not possible. Hence $s=t+1$. Therefore the conductor of
$\omega_{K/F}$ is $t+1$.

Let $\chi$ be a symplectic type character with conductor $m$. Therefore $\chi|_{U_{K}^{m}}=1$ but $\chi|_{U_{K}^{m-1}}\neq 1$.
Again $\chi|_{F^\times}=\omega_{K/F}$, then from relation (\ref{eqn 6.11}) we get $\omega_{K/F}$ is trivial on 
$U_{F}^{\frac{m}{2}}$ for $m$ even, and trivial on $U_{F}^{\frac{m+1}{2}}$ for $m$ odd. These gives use 
$m\geq2(t+1)$ (when $m$ even) and $m\geq 2t+1$ (when $m$ odd). Now we have to prove that when $m$ is odd, and it is exactly 
$2t+1$, i.e., $a(\chi)=2t+1$. Suppose the conductor of $\chi$ is $2t+r$ where $r\geq 3$ and co-prime to $2$. 
Therefore $\chi|_{U_{K}^{2t+r}}=1$ but $\chi|_{U_{K}^{2t+r-1}}\neq 1$. Here $r$ can be written as $r=2b+1$, where $b\geq 1$.
 We can write 
 \begin{center}
  $U_{K}^{2t+r-1}=U_{K}^{2t+2b}=U_{F}^{t+b}U_{K}^{2t+2b+1}=U_{F}^{t+b}U_{K}^{2t+r}$.
 \end{center}
Now if we restrict $\chi$ to $U_{K}^{2t+r-1}$ we have
\begin{align}
 \chi|_{U_{K}^{2t+r-1}}=\chi|_{U_{F}^{t+b}U_{K}^{2t+r}}=\chi|_{U_{F}^{t+b}}\times\chi|_{U_{K}^{2t+r}}
 =\omega_{K/F}|_{U_{F}^{t+b}}=1,\label{eqn 3.12}
\end{align}
since $a(\omega_{K/F})=t+1$ and $a(\chi)=2t+r$.
But the left hand side of equation (\ref{eqn 3.12}), $\chi|_{U_{K}^{2t+r-1}}\neq1$ which is a contradiction. 
Therefore it is impossible that the conductor of 
$\chi$ (when conductor is odd) is $2t+r (r\geq3)$. 
Thus we conclude that when the conductor of symplectic type character is odd, it is exactly $2t+1$.

This completes the proof. 

\end{proof}

\begin{rem}
 Let $\chi,\chi'$ be two symplectic type characters. Then there exists a nontrivial character 
$\eta:K^\times\to\bbC^\times$ with $\eta|_{F^\times}\equiv 1$ such that 
$$\chi'=\eta\cdot\chi.$$
Since we just notice that $a(\chi'), a(\chi)$ are $\ge \{2t+1, 2l\}$, where $l\ge (t+1)$, the conductor of 
$\eta$ must be a even number greater than $2t+1$.

Again we notice that there exists a symplectic type character $\chi_{odd}$ of $K^\times$ with conductor $2t+1$.
It can be proved all 
symplectic type 
characters $\chi$ can be represented  as follows:
$$\chi=\eta\cdot\chi_{odd},\quad \text{where}\quad\eta:K^\times\to\bbC^\times \quad with \quad \eta|_{F^\times}\equiv 1.$$
Then 
$$a(\chi)=a(\eta\cdot\chi_{odd})=max\{a(\eta),a(\chi_{odd})\}=a(\eta).$$
 
\end{rem}

\vspace{.5cm}
\begin{proof}[{\bf Proof of Theorem \ref{Theorem 1.2}}]
By our choice
$\pi_K$ is a uniformizer of $K$ and  $\pi_F=N_{K/F}(\pi_K)$ is a uniformizer of $F$.
This implies $\pi_F=-\pi_{K}^{2}$.
Let $t$ be the ramification break of the extension $K/F$.
We also know from Theorem \ref{Theorem 1.1} that $2d\geq 2(t+1)$, here $t$ is odd and $t<2e$, where 
$e$ is the absolute ramification index of $F$. Therefore when conductor of symplectic type 
character is $2d$, then $d\geq 2$.

The main ingredient of the proof is Theorem \ref{Theorem 2.1}. Here the conductor of $\chi$ is $2d$. Then we 
are in the {\bf even conductor} situation of Theorem \ref{Theorem 2.1}, 
that is, we need to use Corollary \ref{Corollary 2.2}(1).
To do this first we have to choose a $c'\in K$ such that we are able to use Corollary \ref{Corollary 2.2}(1) for 
our symplectic character $\chi$.\\
{\bf Proper choice of $c'$ for the proof:}\\
In general, the idea of choosing proper $c$ for using Theorem \ref{Theorem 2.1} is as follows. Let $\theta$ be a
multiplicative character of $F$ with conductor $a(\theta)\ge 2m$, where $m\in\bbZ_{\ge 0}$, and $\psi_F$ be a nontrivial
additive character of $F$. Now one can prove that
$$\psi_\theta(y):=\theta(1+y)\quad \text{for all $y\in P_F^{a(\theta)-m}$}$$
is a nontrivial character of $P_F^{a(\theta)-m}$ because 
for $y,y'\in P_F^{a(\theta)-m}$, we can write 
 $$\theta(1+y)\theta(1+y')=\theta((1+y+y')(1+\frac{yy'}{1+y+y'}))=\theta(1+y+y'),i.e., 
 \psi_\theta(y+y')=\psi_\theta(y)\cdot \psi_\theta(y')$$
 because $1+\frac{yy'}{1+y+y'}\in U_{F}^{2(a(\theta)-m)}\subseteq U_F^{a(\theta)}$.
 That is, $y\mapsto \theta(1+y)$ is a character of the additive group $P_F^{a(\theta)-m}$.
 This character extends to an additive character of the field $F$ and, by {\bf local additive duality}, 
 there is some $c\in F^\times$ such that
 $$\psi_\theta(y)=\theta(1+y)=\psi_F(c^{-1}y)=(c^{-1}\psi_F)(y),\quad\text{for all $y\in P_F^{a(\theta)-m}$}.$$
So comparing the conductors of both sides we must have:
 $$a(\theta)=-n(c^{-1}\psi)=\nu_F(c)-n(\psi_F),$$
hence {\bf $\nu_F(c)=a(\chi)+n(\psi_F)$ is the right assumption for equation (\ref{eqn 6.0.9}).}

Moreover, every $x\in U_F/U_F^{a(\theta)}$ can be expressed as $x=z(1+y)$, where $y\in P_F^{a(\theta)-m}$ and $z$ 
 runs over the 
 system of representatives for $U_F/U_F^{a(\theta)-m}$, then we will have
  \begin{equation}\label{eqn 5.4.7}
  \sum_{x\in U_F/U_F^{a(\theta)}}\theta^{-1}(x)\psi_F(c^{-1}x)=
  q_F^m \sum_{z\in U_F^{m}/U_F^{a(\theta)-m}}\theta^{-1}(z)\psi_F(c^{-1}z).
 \end{equation}
By using this equation (\ref{eqn 5.4.7}), one also can see that the formula (\ref{eqn 6.0.9}) 
does not depend on the {\bf unit} part of 
$c$ (cf. Remark 2.1).

Now we come to the proof of Theorem 1.2:\\
By our assumption, conductor of $\chi$ is $a(\chi)=2d$, an even positive integer.
Here $m=d$. Then the idea is to use Corollary \ref{Corollary 2.2}(1), so we have 
to choose proper $c'$ here which we do as above and our proper (in the sense the expression of $\epsilon(\chi,\psi)$
is good enough for proving our assertion)
{\bf choice of $c'$ is: $c'=\pi_K^{2d}$}. So now we have to verify whether $c'=\pi_K^{2d}$ satisfies the condition of 
Lamprecht-Tate Theorem \ref{Theorem 2.1} or not. Now we verify this in the following.

We have observed that the Lamprecht-Tate formula (\ref{eqn 6.0.9}) 
is valid for any $c'$ such that $\nu_K(c')=a(\chi)+n(\psi)=2d+0=2d.$
But it order to make the expression $\epsilon(\chi,\psi)$ as in our assertion, we have to improve the choice of $c'$.
And to make Lamprecht-Tate works we need a $c'$ such that
\begin{equation}\label{eqn 3.5}
\chi(1+x)  = (c'^{-1}\psi)(x)\quad\text{for all $x$  such that $ \nu_K(x) \ge a(\chi)/2=d$}.
\end{equation}
Here we see that our $c'=\pi_K^{2d}$ satisfies equation (\ref{eqn 3.5}). Now we are left to show that 
$a(\chi)=-n(c'^{-1}\psi)$, and it can be seen as follows:
$$n(c'^{-1}\psi)=\nu_K(c'^{-1})+n(\psi)=\nu_K(\pi_K^{-2d})+0=-2d=-a(\chi)$$
because conductor of $\chi$ (resp. $\psi$) is $2d$ (resp. $0$). 

Therefore here we can choose $c'=\pi_K^{2d}$ for applying Lamprecht-Tate formula. This is the verification of the 
choice of $c'=\pi_K^{2d}$.\\

We have verified above that our choice $c'=\pi_K^{2d}$ is a proper choice for the Lamprecht-Tate formula \ref{eqn 6.0.9}, and 
here $a(\chi)=2d$, so from Corollary \ref{Corollary 2.2}(1) (even conductor case) we have
 \begin{align*}
 \epsilon(\chi,\psi)
 &=\chi(c')\times\psi(c'^{-1})\\
 &=\chi(\pi_{K}^{2d})\times\psi(\pi_{K}^{-2d})\\
 &=\chi(\pi_{K}^{2})^{d}\times\psi((\pi_{K}^{2})^{-d})\\
 &=\chi(- \pi_F)^{d}\times\psi((-\pi_F)^{-d})\quad\text{since $\pi_F=N_{K/F}(\pi_K)=-\pi_K^2$}\\
 &=\chi(-1)^{d}\omega_{K/F}(\pi_F)^{d}\times 1 \quad \text{since $\psi|_{F}=1$}\\
 &=\chi(-1)^{d}, \quad\text{since $\pi_F\in N_{K/F}(K^\times)$, then $\omega_{K/F}(\pi_F)=1$}.\\
\end{align*}

\end{proof}

\vspace{.2cm}
\begin{proof}[{\bf Proof of Theorem \ref{Theorem 1.3}}]
The idea of this proof is same as Theorem \ref{Theorem 1.2}, but here we have to deal with odd conductor situation
of Lamprecht-Tate formula \ref{eqn 6.0.9}, hence we will use Corollary \ref{Corollary 2.2}(2).

By our choice $\pi_K$ is a uniformizer of $K$, and $N_{K/F}(\pi_K)=\pi_F$ is a uniformizer of $F$.
 Since $a(\chi)=2t+1$ and $n(\psi)=2l+1$, therefore we  choose (it can be checked in the same way what we have done 
 in the proof of Theorem \ref{Theorem 1.2} that $c'$ is suitable for applying Lamprecht-Tate formula)
 $$c'=\pi_{K}^{a(\chi)+n(\psi)}=\pi_{K}^{2t+1+2l+1}=\pi_{K}^{2(t+l+1)}.$$
 
 Moreover, $\psi|_{F}=1$ by Lemma \ref{Lemma 3.1}, so we can write
 $$\psi(c'^{-1})=\psi(\pi_{K}^{-2(t+l+1)})=\psi((-\pi_F)^{-(t+l+1)})=1,$$ 
  and 
\begin{align*}
 \chi(c')=\chi(\pi_{K}^{2(t+l+1)})=\chi(\pi_{K}^{2})^{t+l+1} =\chi(-\pi_F)^{t+l+1}\\
 =\chi(-1)^{t+l+1}
 \times\omega_{K/F}(\pi_F)^{t+l+1}=\chi(-1)^{t+l+1}.
\end{align*}
 
Now from Corollary \ref{Corollary 2.2}(2)  we can write 
\begin{align*}
 \epsilon(\chi,\psi)
 &=\chi(c')\psi_{-1}(c'^{-1})\cdot G(Q)\\
 &=\chi(-1)^{t+l+1}G(Q)\\
 &=\chi(-1)^l G(Q),
\end{align*}
where 
$$G(Q):= q_K^{-1/2}\sum_{x\in P_K^t/P_K^{t+1}} Q(x), \quad Q(x):=\chi^{-1}(1+x)(c'^{-1}\psi)(x).$$

\end{proof}

\begin{conj}
 With the above notations, $G(Q)$ is $\pm 1$.
\end{conj}

\begin{rem}
 Let $\kappa_K$ be the residue field of $K$. In the above theorem,
due to the choice of $c'$, we observe that $Q(x)$ is
a function on $P_K^t/P_K^{t+1}$ such that:
$$  \frac{Q(x+y)}{Q(x)Q(y)}=\chi\left(1+\frac{xy}{1+x+y}\right) = (c'^{-1}\psi)(xy),\quad \textrm{for}\; x,y\in P_K^t.$$

Since $a(\chi)=2t+1 \ne 1$ and $p=2$ hence $\epsilon(\chi,\psi)=\chi(c')\psi(c'^{-1})\cdot G(Q),$\\
for $c'\in K^\times$, which means that $Q(x)=\chi^{-1}(1+x)(c'^{-1}\psi)(x)$ is a 
function on $P_K^t/P_K^{t+1}$ such that $\ve\in\kappa_K\mapsto \ol Q(\ve):=Q(\pi_K^t\ve)$ satisfies
$$  \frac{\ol Q(\ve_1+\ve_2)}{\ol Q(\ve_1)\ol Q(\ve_2)} = (\pi_K^{2t}c'^{-1}\psi)(\ve_1\ve_2).$$
Therefore $\ol Q$ is (in the sense of Subsection 1.1.5 of \cite{Ge 75}) a quadratic character of 
$\kappa_K$ which is related
to the additive character $\tau:= \pi_K^{2t} c'^{-1}\psi,$ and then by the second part of Lemma 3.1 in Subsection 1.1.9 of 
\cite{Ge 75}, we have\\
$G(Q) = \gamma(\ol Q)$ is an 8th root of unity such that:
$$  \gamma(\ol Q)^2 = \ol Q(c'(\tau)),\qquad \gamma(\ol Q)^4 = (-1)^{[\kappa_K:\F_2]},$$
where $c'(\tau)\in\kappa_F^\times$ is determined by the relation $\tau(\ve) = (-1)^{Tr(\ve/ c'(\tau)^2)}$ for all 
$\ve$, and $Tr:=Tr_{\kappa_K|\F_2}.$
\end{rem}

\begin{rem}
When $K/F$ is a quadratic unramified or tamely ramified, we have the explicit signs of epsilon factor of 
symplectic type characters due to Dipendra Prasad 
(see \cite{DP}, pp. 22-23). He uses Deligne's twisting formula
(cf. Corollary \ref{Corollary 2.2}(3)). One also can determine the explicit sign directly by using
Lamprecht-Tate formula. 
Let $K/F$ be a quadratic unramified extension of local field $F$. 
Let $\chi$ be a ramified symplectic type character of $K^\times$, i.e., conductor $a(\chi)\geq 1$
and $\psi_{b}$ be a nontrivial additive character
of $K$ of the form $\psi_{b}(x):=\psi_{K}(bx)=\psi_{F}(\mathrm{Tr}_{K/F}(bx))$ for all $x\in K$ and $b\in K^\times$
with $\mathrm{Tr}_{K/F}(b)=0$. Then we have (cf. \cite{DP}, p. 22-23)
\begin{equation}
 \epsilon(\chi,\psi_{b})=\begin{cases}
                      1 & \text{when $a(\chi)$ even},\\
                      -1 & \text{when $a(\chi)$ odd}.
                     \end{cases}
\end{equation}

If $K/F$ is tamely ramified, then $t=0$. Then from Theorem \ref{Theorem 1.1} we can see 
that the conductor of symplectic type character is \textbf{one or even} and 
the conductor of $\omega_{K/F}$ is $1$, i.e.,
$\omega_{K/F}$ is trivial on 
$1+\pi_F O_{F}$. 
In the tame case we choose $\pi_K$ as a uniformizer and then $\pi_F=N_{K/F}(\pi_K)$ is a uniformizer of $F$.
Now fix a character $\tilde{\omega_{K/F}}$ of $K^\times$ which extends the character $\omega_{K/F}$ of $F^\times$.
We use this isomorphism $U_{K}/U_{K}^{1}\cong U_F/U_{F}^{1}$ (since $K/F$ is a quadratic 
tamely ramified) to extend $\omega_{K/F}$ to $U_K$.
And we know $a(\omega_{K/F})=1$, then it can be seen
that the conductor of $\tilde{\omega_{K/F}}$ is also $1$.
 If  $\chi$ is a character which extends the character $\omega_{K/F}$ with $a(\chi)=1$,
 then $\chi$ is either $\tilde{\omega_{K/F}}$ or $\tilde{\omega_{K/F}}\mu$
 where $\mu$ is an unramified character of $K^\times$ taking value $-1$ at $\pi_K$. Then we have,
 \begin{center}
  $\epsilon(\tilde{\omega_{K/F}}\mu,\psi)=\mu(\pi_K)\epsilon(\tilde{\omega_{K/F}},\psi)=-\epsilon(\tilde{\omega_{K/F}},\psi)$,
 \end{center}
 where $\psi$ is a nontrivial additive character of $K$ with conductor zero.
By using the quadratic classical Gauss sum formula (cf. \cite{LN}, p. 199, Theorem 5.15) 
we can give more explicit formula of $\epsilon(\tilde{\omega_{K/F}},\psi_{-1})$, where
$\psi_{-1}$ is non trivial additive character of $K$ with conductors $-1$. 
 Then 
 \begin{equation}
  \epsilon(\tilde{\omega_{K/F}},\psi'_{-1})=(-1)^{s-1} \quad\text{when $p\equiv 1\pmod{4}$, and $q_F=p^s$}.
 \end{equation}
And when $p\equiv 3\pmod{4}$ and $q_F=p^{2r}(r\geq 1)$,
we have
\begin{equation}
\epsilon(\tilde{\omega_{K/F}},\psi'_{-1})=(-1)^{r-1}. 
\end{equation}
  
 Now let $\chi$ be a symplectic type character of $K^\times$ with conductor 
 $2d (d\geq 1)$. Let $\psi$ be a nontrivial additive character of $K$ which is trivial on $F$ with
 conductor $zero$.
 Then by using Theorem \ref{Theorem 1.2} we have
 \begin{equation}
  \epsilon(\chi,\psi)=\begin{cases}
    1 &\text{if $q_F\equiv 1\pmod{4}$}\\
    (-1)^{d} & \text{if $q_F\equiv 3\pmod{4}$}.
   \end{cases}
 \end{equation}

\end{rem}

\vspace{20mm}
\textbf{Acknowledgements:} 
\thispagestyle{empty}
I express my sincere gratitude to E.-W. Zink  for his valuable comments on the paper.
I also extend my gratitude to the referee for his/her comments for the improvement of the paper.


\end{document}